\documentclass[11pt]{article}

\usepackage[T1]{fontenc}
\usepackage{lmodern}
\usepackage[a4paper,margin=29mm]{geometry}
\usepackage{amsmath,amssymb,amsthm}
\usepackage{microtype}
\usepackage[colorlinks=true,linkcolor=blue,citecolor=blue,urlcolor=blue]{hyperref}
\newtheorem{theorem}{Theorem}[section]
\newtheorem{lemma}[theorem]{Lemma}
\newtheorem{proposition}[theorem]{Proposition}

\title{A note on the logarithmic exponent of Ramsey number \texorpdfstring{$r(4,t)$}{r(4,t)}}
\author{
Kaizhe Chen\thanks{Department of Mathematics, Princeton University, Princeton, USA. Email: kc8556@princeton.edu}
}
\date{October 9, 2026}

\begin{document}
\maketitle

\begin{abstract}
In this short note, we point out that the recent lower bounds for off-diagonal Ramsey numbers $r(s,t)$ with fixed $s\ge 5$, together with the recent proof of the Ajtai--Erd\H{o}s--Koml\'os--Szemer\'edi independence conjecture, imply
\[
  r(4,t)=\frac{t^3}{(\log t)^{2+o(1)}}.
\]
This note is solely for informal circulation and will not be submitted for publication.
\end{abstract}

\section{Introduction}

For integers $s,t\ge2$, let $r(s,t)$ be the least integer $N$ such that every graph on $N$ vertices contains either a clique of size $s$ or an independent set of size $t$. We consider the off-diagonal regime, in which $s$ is fixed and $t\to\infty$. For more on Ramsey theory and recent developments, we refer the reader to the surveys of Conlon, Fox, and Sudakov~\cite{CFS} and Verstra\"ete~\cite{VerstraeteSurvey}.

In 1980, Ajtai, Koml\'os, and Szemer\'edi~\cite{AKS} proved $r(s,t)=O_s(t^{s-1}/(\log t)^{s-2})$ for any fixed $s\ge3$.
Pseudorandom graphs provide another approach to off-diagonal Ramsey numbers; see Mubayi and Verstra\"ete~\cite{MubayiVerstraete}.
Using a finite-geometric construction, Mattheus and Verstra\"ete~\cite{MattheusVerstraete} established
\[
  r(4,t)\ge \Omega\!\left(\frac{t^3}{(\log t)^4}\right),
\]
leaving a gap of two in the logarithmic exponent. Recently, Brada\v{c}~\cite{Bradac} proved that, for every fixed integer $s\ge 3$,
\[
  r(s,t)\ge \Omega\left(\frac{t^{s-1}}{(\log t)^{2s-4}}\right),
\]
which determines the polynomial exponent of $r(s,t)$ for every fixed $s$. The recent preprints~\cite{OpenAI5,OpenAIhigh} determine the logarithmic exponent of $r(s,t)$ for every fixed $s\ge5$.

\begin{lemma}[{\cite{OpenAI5}}]\label{five}
As $t\to\infty$,
\begin{equation*}
  r(5,t)\ge\frac{t^4}{(\log t)^{3+o(1)}}.
\end{equation*}
\end{lemma}

\begin{lemma}[{\cite{OpenAIhigh}}]\label{high}
For every fixed integer $s\ge6$, as $t\to\infty$,
\begin{equation*}
  r(s,t)\ge\frac{t^{s-1}}{(\log t)^{s-2+o(1)}}.
\end{equation*}
\end{lemma}

For a graph $G$, we write $\alpha(G)$ for its independence number.
The recent preprint~\cite{OpenAIColoring} proves the Alon--Krivelevich--Sudakov coloring conjecture, stating that for every fixed graph $F$, every $F$-free graph $G$ with sufficiently large maximum degree $\Delta$ has chromatic number at most $O_F(\frac{\Delta}{\log\Delta})$.
In particular, this implies the Ajtai--Erd\H{o}s--Koml\'os--Szemer\'edi independence conjecture, as stated below.

\begin{lemma}[\cite{AIIndependence}]
\label{lem:shearer}
For any fixed integer $s\ge 4$, there exists a constant $c_s>0$ such that every $K_s$-free graph $G$ on $n$ vertices with average degree $d\ge 2$ satisfies
\[ \alpha(G)\ge\frac{c_sn\log d}{d}. \]
\end{lemma}

We observe that combining these results determines the logarithmic exponent of $r(4,t)$.

\begin{theorem}\label{thm:main}
As $t\to\infty$,
\begin{equation*}
  r(4,t)=\frac{t^3}{(\log t)^{2+o(1)}}.
\end{equation*}
\end{theorem}

\section{A transfer principle}

\begin{proposition}\label{prop}
For any fixed integer $s\ge4$, there exists a constant $a_s>0$ such that, for sufficiently large $t$,
\begin{equation}\label{eq:transfer-output}
  r(s-1,t)\ge \frac{a_s \log t}{t}r(s,t).
\end{equation}
\end{proposition}

\begin{proof}
Let $n=r(s,t)-1$. There exists a $K_s$-free graph $G$ on $n$ vertices with $\alpha(G)<t$.
Write $d$ for its average degree. By the Caro--Wei bound (see e.g. \cite[p.100]{AlonSpencer}), we have $\alpha (G) \ge \frac{n}{d+1}$. Consequently,
\[
  d+1\ge \frac{n}{\alpha (G)}>\frac{n}{t}.
\]
The classical lower bound (see e.g. \cite{Spencer}) implies $\frac{r(s,t)}{t^2}\to\infty$, so $\frac{n}{t^2}\to\infty$, and hence $d\ge t$. Lemma~\ref{lem:shearer} then gives
\[
  t>\alpha(G) \ge \frac{c_sn\log d}{d} \ge \frac{c_sn\log t}{d}.
\]
That is,
\[ 
   d >\frac{c_sn\log t}{t} = \frac{c_s\log t}{t}(r(s,t)-1).
\]
Choose a vertex $v$ of degree at least $d$. The subgraph of $G$ induced by all the neighbors of $v$ is $K_{s-1}$-free and has independence number less than $t$. Hence $r(s-1,t)>d$, completing the proof.
\end{proof}

Applying Proposition~\ref{prop} to either Lemma~\ref{five} or Lemma~\ref{high} gives Theorem~\ref{thm:main}. Moreover, applying Proposition~\ref{prop} to Lemma~\ref{high} yields Lemma~\ref{five}.
Since Proposition~\ref{prop} incurs only a constant loss, removing the $o(1)$ from the logarithmic exponent in the lower bound on $r(s,t)$, for any fixed $s\ge 5$, would also remove it from the corresponding bound on $r(s-1,t)$.

\section*{Acknowledgements}
The author would like to thank Noga Alon for pointing out that the coloring theorem in~\cite{OpenAIColoring} allows us to avoid the $\log\log t$ loss incurred by using Shearer's bound~\cite{Shearer}.

\end{document}